\documentclass[11pt]{article}

\usepackage[T1]{fontenc}
\usepackage[utf8]{inputenc}
\usepackage{lmodern}
\usepackage{microtype}
\usepackage[a4paper,margin=29mm]{geometry}
\usepackage{amsmath,amssymb,amsthm,mathtools}
\usepackage{enumitem}
\usepackage{booktabs}
\usepackage{hyperref}
\usepackage{color}
\usepackage{float}
\hypersetup{
  hidelinks,
  pdftitle={The Post Correspondence Problem for Free Groups Is Undecidable},
  pdfauthor={Andre Carvalho}
}
\allowdisplaybreaks
\usepackage{graphicx}
\usepackage{tikz}
\usetikzlibrary{arrows.meta,positioning,calc}
\newtheorem{theorem}{Theorem}[section]
\newtheorem{proposition}[theorem]{Proposition}
\newtheorem{lemma}[theorem]{Lemma}
\newtheorem{corollary}[theorem]{Corollary}
\theoremstyle{definition}

\newtheorem{remark}[theorem]{Remark}

\newcommand{\eps}{\varepsilon}
\newcommand{\Eq}{\operatorname{Eq}}
\newcommand{\first}{\operatorname{first}}
\newcommand{\inlab}{\operatorname{in}}
\newcommand{\outlab}{\operatorname{out}}
\newcommand{\red}[1]{\overline{#1}}

\newcommand{\wt}{\widetilde}

\newcommand{\Z}{\mathbb{Z}}

\newcommand{\End}{\text{End}}

\newcommand{\Fix}{\mathrm{Fix}}

\newcommand{\mc}{\mathcal}
\newcommand{\T}{\mathcal T}

\title{The Post Correspondence Problem for free groups is undecidable}
\author{Andr\'e Carvalho\\
\small Research Center in Mathematics and Applications (CIMA)\\
\small Department of Mathematics, School of Sciences and Technology of the University of \'Evora\\
\small Rua Rom\~ao Ramalho, 59, 7000--671 \'Evora, Portugal\\
\small \texttt{andrecruzcarvalho@gmail.com}}
\date{}

\begin{document}
\maketitle

\begin{abstract}
	We prove that the Post Correspondence Problem for finitely generated free groups is undecidable, even when one of the two homomorphisms is injective and has finite-index image. This resolves a longstanding open problem in algorithmic group theory. 

	The proof proceeds through a   connection with finite-state transducers. Given a cyclic tag system $\mathcal C$, we effectively construct a finite partial deterministic inverse transducer $\mathcal T_{\mathcal C}$ whose fixed-point set is nontrivial if and only if $\mathcal C$ halts. We then associate to any such transducer two homomorphisms
	\(
	g,h\colon F_Y\longrightarrow F_A,
	\)
	with $h$ injective, such that their equalizer is nontrivial precisely when the transducer has a nontrivial fixed loop. As an immediate consequence,  the rank of these
	equalizers cannot be computed in general, answering a question posed by Stallings in 1984.

We further prove that there is no algorithm which decides whether the fixed subgroup of a virtual endomorphism of a finitely generated free group is trivial. Finally, we apply the main result to show that the stabilizer problem is undecidable for free subgroups of \(\operatorname{SL}_4(\mathbb Z)\), and that the upper-right-corner problem is undecidable for free subgroups of \(\operatorname{SL}_5(\mathbb Z)\), even when the given generators are promised to form a free basis, improving on recent results of Breuillard and Kocharyan.
\end{abstract}

\section{Introduction}

The Post Correspondence Problem is one of the classical sources of
undecidability in theoretical computer science. Given two homomorphisms
\(
\phi,\psi\colon X^*\longrightarrow Y^*
\)
between finitely generated free monoids, it asks whether there exists a
nonempty word \(w\in X^*\) such that
\(
w\phi=w\psi.
\)
Equivalently, the problem asks whether the equalizer
\[
\Eq(\phi,\psi)
=
\{w\in X^*\mid w\phi=w\psi\}
\]
is nontrivial. Post proved in 1946 that no algorithm decides this problem
\cite{[Pos46]}.

The analogous problem for free groups has remained open. Given finite
alphabets \(A\) and \(B\) and homomorphisms
\(
\phi,\psi\colon F_A\longrightarrow F_B,
\)
the \emph{Post Correspondence Problem for free groups} asks whether
\[
\Eq(\phi,\psi)
=
\{x\in F_A\mid x\phi=x\psi\}
\]
contains a nontrivial element. Although   progress has been made
for other classes of groups \cite{[MNU14],[CLL24]}, the problem for free
groups has been regarded as an important open question in
algorithmic group theory (see
\cite[Problem~5.1.4]{[DKMM19]},
\cite[Section~1.4]{[MNU14]}, \cite[Problem (F41)]{[KMS26]} and the survey \cite{[CL21]}).

The main result of this paper settles the problem in a stronger form.

\newtheorem*{mainintro3}{Theorem A}
\begin{mainintro3}
	There is no algorithm which, given homomorphisms
	\(
	g,h\colon F_Y\longrightarrow F_A
	\)
	between finitely generated free groups, with $h$ injective and with finite-index image, decides whether
	\(
	\Eq(g,h)\neq\{1\}.
	\)
 
\end{mainintro3}

Goldstein and Turner proved that if at least one of two homomorphisms
between finitely generated free groups is injective, then their equalizer
is finitely generated \cite{[GT86]}. Thus, the equalizers occurring in the 
theorem above are not  infinitely generated
subgroups. They are finitely generated free groups, but there is no
algorithm that determines whether their rank is zero. It follows immediately that there is no algorithm which computes the rank or a free basis of $\Eq(g,h)$. In particular, this gives a negative answer to Stallings's Rank and Basis Problems \cite{[Sta87]}.
 
	The undecidability of triviality remains valid when the target group is
	\(F_2\), and it also remains valid for pairs of endomorphisms of a single
	finitely generated free group, one of which is injective.

\paragraph{Fixed subgroups and equalizers.}

The result contrasts sharply with the theory of fixed subgroups. If
\(
\alpha\colon F_A\longrightarrow F_A
\)
is an endomorphism, then
\[
\Fix(\alpha)
=
\{x\in F_A\mid x\alpha=x\}
=
\Eq(\alpha,\operatorname{id})
\]
is a special case of an equalizer.

The study of fixed subgroups has a long history. Gersten and Cooper independently proved that the fixed subgroup of an
automorphism of a finitely generated free group is finitely generated
\cite{[Ger87],[Coo87]}. Bestvina and Handel subsequently developed the theory of train tracks and proved Scott's
conjecture \cite{[BH92]}: if \(\alpha\) is an automorphism of a free group of rank \(n\),
then
\(
\operatorname{rank}\Fix(\alpha)\leq n.
\)
 Imrich and Turner introduced stable images and extended
these finite-generation and rank results to arbitrary endomorphisms
\cite{[IT89]}.

The corresponding algorithmic problem is more complicated.
Bogopolski and Maslakova constructed an algorithm computing a basis of the
fixed subgroup of an automorphism \cite{[BM16]}. More recently, Mutanguha
proved that stable images of free-group endomorphisms are computable
\cite{[Mut22]}, allowing the Bogopolski--Maslakova algorithm to be applied
to arbitrary endomorphisms. Consequently, fixed subgroups of free-group
endomorphisms are effectively computable and, in particular, their
triviality is decidable. We refer to \cite{[Ven14]} for a survey.

The case of equalizers is much more complicated. When neither of the two maps is
injective, the equalizer may fail to be finitely generated. If one map is
injective, finite generation is given by the theorem of Goldstein and
Turner \cite{[GT86]}, but the corresponding rank and basis problems have
remained open in general.

Stallings asked whether the rank of the equalizer of two injective
endomorphisms of a free group of rank \(n\) is bounded above by \(n\).
This was recently disproved by Lei and Zhang \cite{[LZ26]}. Stallings also asked
whether the rank of an equalizer of two homomorphisms between finitely
generated free groups is computable \cite{[Sta87]}; we refer to this as
the \emph{Rank Problem}. The problem of computing a free basis of the
equalizer, the \emph{Basis Problem}, was shown in \cite{[CL21]} to be
equivalent to the Rank Problem.

Positive results are known for several restricted classes of homomorphisms
\cite{[BM16],[CL20],[CL20b],[FH18]}. In rank two, Logan proved that
finitely generated equalizers have rank  at most two \cite{[Log22]}. Related questions concerning
uniformly continuous extensions of endomorphisms to the boundary have also
been studied in \cite{[CS26]}. Our theorem shows that
these positive results cannot extend to arbitrary finitely generated free
groups.

\paragraph{Virtual endomorphisms and complete transducers.} Let $G$ be a group and $H$ be a finite index subgroup of $G$. We say that a homomorphism from $H$ to $G$ is a virtual endomorphism of $G$. 
 We also prove that there is no algorithm which, given a finite-index subgroup
\(H \leq F\) and a homomorphism
\(
\theta \colon H \longrightarrow F,
\)
decides whether \(\operatorname{Fix}(\theta)\) is trivial, showing that virtual endomorphisms are much harder computationally than endomorphisms. Using a transducer
associated with \(H\), we further show that the rationality theorem for fixed
points of complete inverse transducers obtained by Silva \cite[Theorem 3.2]{[Sil13]} is necessarily non-effective: no algorithm
constructs a finite automaton recognizing their reduced fixed-point language. In fact, no algorithm decides triviality of the fixed-point set of such a transducer. Concretely, we show the following result:

\newtheorem*{virtualtransducertheorem}{Theorem C}
\begin{virtualtransducertheorem}
	The following problems are undecidable.
	\begin{enumerate}
		\item[\textup{(i)}]
		Given a finitely generated free group \(F\), a finite-index subgroup
		\(H\leq F\), and a homomorphism
		\(
		\theta\colon H\longrightarrow F,
		\)
		decide whether
		\(
		\operatorname{Fix}(\theta)=\{1\}.
		\)
		
		\item[\textup{(ii)}]
		Given a finite complete deterministic inverse transducer
		\(\mathcal T\), decide whether
		\[
		\operatorname{Fix}(\widetilde{\mathcal T})=\{1\}.
		\]
	\end{enumerate}
	Consequently, there is no algorithm which constructs, from a finite complete
	deterministic inverse transducer \(\mathcal T\), a finite automaton
	recognizing the reduced words representing the elements of
	\(
	\operatorname{Fix}(\widetilde{\mathcal T}).
	\)
\end{virtualtransducertheorem}

\paragraph{Algorithmic problems in integral matrix groups.} The free monoid version of the PCP has been used with great success to obtain undecidability results in different areas.  Following this work, we show that the upper-right corner problem is undecidable for free subgroups of $\operatorname{SL}_5(\Z)$ and that the stabilizer problem is undecidable for free subgroups of $\operatorname{SL}_4(\Z)$, improving recent results of Breuillard and Kocharyan \cite{[BK25]}.

\newtheorem*{matrixconsequencestheorem}{Theorem D}
\begin{matrixconsequencestheorem}
	The following problems are undecidable.
	\begin{enumerate}
		\item[\textup{(i)}]
		Given a finite tuple of matrices promised to be a free basis of a subgroup
		\(
		G\leq\operatorname{SL}_4(\mathbb Z),
		\)
		decide whether
		\[
		\operatorname{Stab}_G(e_1)
		=
		\{M\in G\mid Me_1=e_1\}
		\]
		contains a nonidentity element.
		
		\item[\textup{(ii)}]
		Given a finite tuple of matrices promised to be a free basis of a subgroup
		\(
		G\leq\operatorname{SL}_5(\mathbb Z),
		\)
		decide whether \(G\) contains a nonidentity matrix whose
		\((1,5)\)-entry is zero.
	\end{enumerate}
\end{matrixconsequencestheorem}

The stabilizer result lowers the previously known integral dimension bound
from nine to four. The upper-right-corner result lowers the corresponding
dimension bound from nine to five and, at the same time, establishes
undecidability over the integers, answering a question from \cite{[BK25]}. These applications show that the
free-group Post Correspondence Problem can be used as a
source of undecidability in  matrix groups, playing a similar role to the monoid PCP.

\paragraph{Outline of the proof of the main result.}

The proof reduces the halting problem of cyclic tag systems to  deciding triviality of the fixed-point set for finite inverse
transducers and that to an instance of PCP for free groups.
\begin{align*}
&\text{halting of cyclic tag systems}\\
\ \leq\
&\text{existence of nontrivial fixed points of finite inverse transducers}\\
\ \leq\
&\text{nontriviality of equalizers of free-group homomorphisms}.
\end{align*}

The first reduction is the main construction of the paper. A finite partial
deterministic inverse transducer is a finite-state device which reads words
over an alphabet with formal inverses and produces output words, with every
transition accompanied by the corresponding inverse transition. After free
reduction of the output, such a transducer $\T$ defines a partial map $\wt \T$ on a
finitely generated free group. Inverse transducers have previously been
used in the study of fixed points of free-group maps
\cite{[Sil13]}.

Given a cyclic tag system \(\mathcal C\), we effectively construct a finite
partial deterministic inverse transducer
\(\mathcal T_{\mathcal C}\). The construction uses the discrepancy
\[
\Delta(w)
=
w^{-1}\bigl(w\widetilde{\mathcal T}_{\mathcal C}\bigr)
\]
between the input read so far and its output. Along a distinguished path in
the transducer, the freely reduced discrepancy acts as the queue of the tag
system. Reading its first binary letter simulates one step of the cyclic
tag system, while auxiliary marker letters record the number of simulated
steps and prevent the discrepancy from collapsing to the identity.

If the tag system halts, its queue eventually contains only marker letters.
These letters can then be cyclically read without changing the discrepancy,
producing a nontrivial fixed loop in the transducer. We also prove
that every nontrivial fixed point must be represented by a closed path. We can show that the existence
of such a loop forces a repetition in the computation. The marker
recording the number of steps excludes such a repetition when the tag
system does not halt.

Since cyclic tag systems are Turing-complete \cite{[Coo04]}, this gives the
following intermediate result.

\newtheorem*{theoremintro2}{Theorem B}
\begin{theoremintro2}
	It is undecidable whether a finite partial deterministic inverse transducer
	has a nontrivial fixed point.
\end{theoremintro2}

More precisely, from a cyclic tag system \(\mathcal C\), we effectively
construct a finite partial deterministic inverse transducer
\(\mathcal T_{\mathcal C}\) such that the following conditions are
equivalent:
\begin{enumerate}
	\item[\textup{(i)}] \(\mathcal C\) halts;
	\item[\textup{(ii)}]
	\(\mathcal T_{\mathcal C}\) has a nontrivial fixed loop at its
	distinguished state \(s\);
	\item[\textup{(iii)}]
	\(
	\Fix\bigl(\widetilde{\mathcal T}_{\mathcal C}\bigr)
	\neq
	\{1\}.
	\)
\end{enumerate}
Moreover, every fixed point of
\(\widetilde{\mathcal T}_{\mathcal C}\) is represented by a closed path at
\(s\). This last assertion is essential: it excludes fixed points arising
from unintended paths and allows the second reduction to use the subgroup
of closed input paths.

The second reduction is algebraic. Let
\(
\mathcal T=(Q,q_0,\delta,\lambda)
\)
be a finite partial inverse transducer, and let
\[
K
=
\{w\in F_A
\mid
w\text{ labels a closed input path at }q_0\}.
\]
The input automaton of \(\mathcal T\) is a finite inverse automaton, and
hence \(K\) is a finitely generated subgroup of \(F_A\). A free basis
\(
B=\{b_1,\ldots,b_r\}
\)
of \(K\) can be computed from the corresponding Stallings automaton.

Let
\(
Y=\{y_1,\ldots,y_r\}.
\)
We define homomorphisms
\[
g,h\colon F_Y\longrightarrow F_A
\]
by letting \(y_i h=b_i\) and letting \(y_i g\) be the output produced by
the transducer while reading the closed path labelled by \(b_i\). Since
\(B\) is a free basis of \(K\), the homomorphism \(h\) is injective.
Moreover,
\[
xg=xh
\]
holds for a nontrivial \(x\in F_Y\) precisely when \(xh\) labels a
nontrivial fixed loop of \(\mathcal T\). Consequently,
\[
\Eq(g,h)\neq\{1\}
\quad\Longleftrightarrow\quad
\mathcal T
\text{ has a nontrivial fixed loop at }q_0.
\]

Combining this construction with Theorem B proves the basic undecidability of the Post Correspondence  Problem; passing to complete transducers in Section 5 yields the finite-index image condition in Theorem A.
The construction is effective at every stage
and preserves the injectivity of one of the two homomorphisms.

 Since
every finitely generated free group embeds in \(F_2\), undecidability
persists for homomorphisms
\[
g,h\colon F_Y\longrightarrow F_2.
\]
By composing with suitable embeddings, the result can also be formulated
for pairs of endomorphisms of a single finitely generated free group. The
construction ensures that one of the maps is injective.

The paper is organized as follows. In Section~2 we introduce cyclic tag
systems, partial inverse transducers, and the basic properties of
discrepancies. In Section~3 we construct the transducer
\(\mathcal T_{\mathcal C}\) associated with a cyclic tag system and prove
Theorem~B. In Section~4 we encode fixed loops of inverse transducers as
equalizers of free-group homomorphisms and prove  Theorem A. In Section~5 we pass to complete
inverse transducers, prove the finite-index strengthening in Theorem~A,
and establish Theorem~C. In Section~6 we prove the matrix-group
consequences collected in Theorem~D.

 \section{Preliminaries and notation}
 We start by introducing preliminary results about cyclic tag systems and transducers. For further details, the reader is referred to \cite{[Coo04],[Ric08],[Ber79]}. 
 \subsection{Cyclic tag systems}
 A cyclic tag system $\mathcal C$ is a tuple of words $(u_0,\ldots, u_{m-1})$ over $\{0,1\}$, together with a starting word $w_0\in\{0,1\}^*$ and a queue. The queue starts with $w_0$ and evolves according to the following rule: if, at an instant $i$, the queue has value $u=au'$, with $a\in \{0,1\}$, we update it to $u'\alpha_i(a)$, where $\alpha_i(a)=\varepsilon$ if $a=0$ and $\alpha_i(a)=u_{i \text{ mod $m$}}$ if $a=1$. 
 
 We say that the system halts when the queue is empty. We will usually denote by $W_n$ the value of the queue after $n$ steps.
 
 Cook \cite{[Coo04]} showed that cyclic tag systems can emulate any tag system, which are Turing-complete. Hence, the halting problem for cyclic tag systems is undecidable (see also \cite{[Ric08]}). 

\subsection{Transducers}

For a finite alphabet $A$, put
\(
 \widetilde A=A\cup A^{-1}.
\)
A word over $\widetilde A$ represents an element of $F_A$.  For a word $w\in \wt A^*$, let $\red{w}$ denote its freely reduced form, and let $|w|$ denote the length of $\red{w}$.  If $w\neq1$, let $\first(w)$ be the first letter of $\red{w}$.  When there is no risk of confusion, we use the same notation for a word and the element of the free group that it represents.

A finite partial deterministic $A$-transducer is a quadruple $\T=(Q,q_0,\delta,\lambda)$ where $Q$ is a finite set of states, $q_0\in Q$ is the initial state, and $\delta:Q\times A\to Q$ and $\lambda:Q\times A\to A^*$ are partial functions with the same domain.
Both $\delta$ and $\lambda$ can be naturally extended to partial mappings $Q\times A^*\to Q$ and $Q\times A^*\to A^*$, respectively.

A finite partial deterministic $\wt A$-transducer is said to be inverse if 
$$p\xrightarrow{a|w} q \quad\text{ is an edge of $\T$ if and only if } q\xrightarrow{a^{-1}|w^{-1}} p \quad\text{ is an edge of $\T$}.$$

A label can be seen as $\mathrm{input}|\mathrm{output}$. A reduced path is one where the input label is a reduced word, that is, one without backtracking. Naturally, every nontrivial reduced closed path at the basepoint is labelled (on the input side) by a nontrivial reduced word. The following is   an analogue of \cite[Proposition 3.1]{[Sil13]} for partial transducers.
\begin{lemma} 
	Let
	\(
	\mathcal T=(Q,q_0,\delta,\lambda)
	\)
	be a finite partial deterministic inverse $\widetilde A$-transducer, and let
	\(
	\theta:\widetilde A^*\longrightarrow F_A
	\)
	be the natural projection. The partial mappings $\delta$ and $\lambda$ induce  partial mappings $\wt \delta$ and $\wt \lambda$ on $Q\times F_A$ with domain

	\[
	\left\{
	(q,g)\in Q\times F_A
	\;\middle|\;
	\overline g \text{ is readable from }q
	\right\}.
	\]
	Hence
	\[
	g\widetilde{\mathcal T}
	=
	(q_0,g)\widetilde\lambda
	\]
	defines a partial mapping $\widetilde{\mathcal T}:F_A\rightharpoonup F_A$.
\end{lemma}

\begin{proof}
	It is enough to prove that deleting a single freely reducible
	factor does not affect readability, the terminal state, or the output
	as an element of $F_A$.
	
	Suppose that
	\[
	u=xaa^{-1}y,
	\qquad
	x,y\in\widetilde A^*,
	\qquad
	a\in\widetilde A,
	\]
	and that $u$ is readable from $q$. Put
	\(
	p=(q,x)\delta.
	\)
	Suppose that the edge read after $x$ is
	\(
	p\xrightarrow{a\mid U}r.
	\)
	Since $\mathcal T$ is inverse, the edge
	\(
	r\xrightarrow{a^{-1}\mid U^{-1}}p
	\)
	also belongs to $\mathcal T$. Moreover, since $\mathcal T$ is
	deterministic, this is the unique edge leaving $r$ with input label
	$a^{-1}$. It must therefore be the edge traversed by the factor
	$a^{-1}$ in the readable word $u$.
	
	Thus the subpath labelled $aa^{-1}$ returns from $p$ to $p$.
	Consequently, $xy$ is readable from $q$ and
	\(
	(q,u)\delta=(q,xy)\delta.
	\)
	If
	\(
	X=(q,x)\lambda\) and
	\(Y=(p,y)\lambda,
	\)
	then
	\[
	(q,u)\lambda=XUU^{-1}Y,
	\qquad
	(q,xy)\lambda=XY.
	\]
	It follows that
	\[
	\bigl((q,u)\lambda\bigr)\theta
	=
	\bigl((q,xy)\lambda\bigr)\theta.
	\]
	
	Iterating this argument over a sequence of free cancellations taking
	$u$ to $\red{u}$ gives
	\(
	(q,u)\delta=(q,\red{u})\delta
	\)
	and
	\(
	\bigl((q,u)\lambda\bigr)\theta
	=
	\bigl((q,\red{u})\lambda\bigr)\theta.
	\)
	
Hence the induced partial mappings are therefore well-defined.
\end{proof}

\subsubsection{Fixed loops and discrepancies}

Let $\T$ be a finite inverse $\wt A$-transducer. Define $$\Fix(\wt\T)=\{w\in F_A\mid w\wt\T=w\}.$$
 
We say that $w\in F_A$ is a \emph{fixed loop} (at the basepoint) if $w\in \Fix(\wt \T)$ and $(q_0,w)\delta=q_0$.

We define the \emph{discrepancy} of a word $w\in F_A$ that represents a readable path from the basepoint as $w^{-1}w\wt\T$. 

Throughout the paper, when reading a word, we will keep track of the discrepancy of the input as well as the state we are currently at and denote that by a pair $(d,q)\in F_A\times Q$, which we will call a discrepancy pair.

If, when reading a word, the current discrepancy is $d\in F_A$ and an edge labelled $a\mid U$ is crossed, the new discrepancy is
\begin{align}
	d'&=\red{a^{-1}dU}.
	\label{eq:update}
\end{align}
More generally, if we read a path $\gamma$, the discrepancy is
\begin{align}
	d'&=\red{\inlab(\gamma)^{-1}d\,\outlab(\gamma)},
	\label{eq:path-update}
\end{align}
where $\inlab(\gamma)$ is the input label of $\gamma$ and $\outlab(\gamma)$ is the output label.

Starting with discrepancy $1$, a closed path $\gamma$ is a fixed loop if and only if $(q_0,\gamma)\delta=q_0$ and its final discrepancy is $1$.

\section{Construction of a partial transducer from a cyclic tag system}\label{sec:construction}
The purpose of this section is to describe a construction of a finite partial deterministic inverse transducer associated to a cyclic tag system so that the system halts if and only if the transducer has a nontrivial fixed point. Moreover, the transducer will have the property that all fixed points will be, in fact, fixed loops at the initial state.

Fix a cyclic tag system
\(
 \mathcal C=(u_0,\ldots,u_{m-1})
\)
with initial word $w_0$.  Let
\[
 A=\{0,1,\#,H,p\},
 \qquad
 Q=\{s,0,1,\ldots,m-1\}.
\]

When we write $i$ to refer to an element of $Q$ we always mean $i$ mod $m$.  Let $\delta:Q\times   A\rightharpoonup Q$ be defined as:
\begin{itemize}
	\item $(s,\#)\delta=0$;
	\item for $i\in \{0,\ldots, m-1\}$ $(i,0)\delta=(i,1)\delta=i+1$ and $(i,H)\delta=(i,p)\delta=i$;
\end{itemize}
and with the remaining transitions undefined.

Also, let $\lambda:Q\times   A\rightharpoonup A^*$ be defined as:
\begin{itemize}
	\item $(s,\#)\lambda=\#w_0p$;
	\item for $i\in \{0,\ldots, m-1\}$ $(i,0)\lambda=H$, $(i,1)\lambda=u_iH$, $(i,H)\lambda=H$ and $(i,p)\lambda=p$;
\end{itemize}
and with the remaining transitions undefined.
	We now  add all inverse-letter transitions to obtain a partial finite inverse transducer $ \T_{\mc C}$.

\begin{figure}[H]
	\centering
	\resizebox{\linewidth}{!}{%
		\begin{tikzpicture}[
			>=Stealth,
			semithick,
			state/.style={
				circle,
				draw,
				minimum size=9mm,
				inner sep=1pt
			},
			lab/.style={
				fill=white,
				inner sep=1.5pt,
				align=center,
				font=\small
			}
			]
			
			\node[state] (s) {$s$};
			\node[state, right=2.8cm of s] (q0) {$0$};
			\node[state, right=3.0cm of q0] (q1) {$1$};
			\node[state, right=3.0cm of q1] (q2) {$2$};
			
			\node[coordinate, right=1.5cm of q2] (midL) {};
			\node[coordinate, right=1.5cm of midL] (midR) {};
			\node at ($(midL)!0.5!(midR)$) {$\cdots$};
			
			\node[state, right=1.5cm of midR] (qm) {$m-1$};
			
			\draw[->]
			(s) -- node[lab,above] {$\#\mid \#w_0p$} (q0);
			
			\draw[->]
			(q0) to[bend left=15]
			node[lab,above] {$\substack{0\mid H\\[1pt]1\mid u_0H}$}
			(q1);
			
			\draw[->]
			(q1) to[bend left=15]
			node[lab,above] {$\substack{0\mid H\\[1pt]1\mid u_1H}$}
			(q2);
			
			\draw[->]
			(q2) to[bend left=12]
			node[lab,above] {$\substack{0\mid H\\[1pt]1\mid u_2H}$}
			(midL);
			
			\draw[->]
			(midR) to[bend left=12]
			node[lab,above] {$\substack{0\mid H\\[1pt]1\mid u_{m-2}H}$}
			(qm);
			
			\draw[->]
			(qm) to[out=-105,in=-75,looseness=0.75]
			node[lab,below] {$\substack{0\mid H\\[1pt]1\mid u_{m-1}H}$}
			(q0);
			
			\foreach \v/\pos in {q0/above,q1/above,q2/above,qm/above}
			{
				\draw[->]
				(\v) edge[loop \pos,looseness=5.2]
				node[lab] {$\substack{H\mid H\\[1pt]p\mid p}$}
				(\v);
			}
			
		\end{tikzpicture}%
	}
	
	\caption{
		The partial inverse transducer $\mathcal T_{\mathcal C}$
	}

\end{figure}
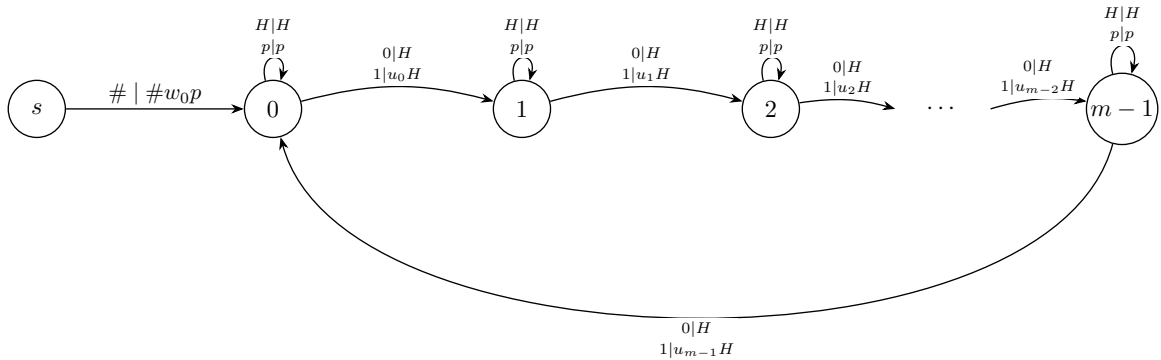

It is easy to check that $\T_{\mc C}$ is indeed a finite partial inverse transducer.

Let $e_{\#}$ denote the edge
\(
s\xrightarrow{\#\mid\#w_0p}0.
\)
Starting from the discrepancy pair $(1,s)$ and reading $e_{\#}$ gives
\[
\left(\red{\#^{-1}(\#w_0p)},0\right)
=(w_0p,0).
\]
Put
$
	d_0=w_0p
$.
Suppose that the current discrepancy pair is $(aR,i)$, where $a$ is the first letter of the   word $aR$, and suppose that the edge leaving $i$ with input label $a$ has output $U$.  Formula~\eqref{eq:update} gives
\begin{equation*} 
	(aR,i)\longmapsto(RU,r),
\end{equation*}
where $r$ is the terminal state of the edge.  Thus the first letter is removed and the output is appended.

Define a homomorphism
\[
\beta:\{0,1,H,p\}^*\longrightarrow\{0,1\}^*
\]
by
\[
\beta(0)=0,
\qquad
\beta(1)=1,
\qquad
\beta(H)=\beta(p)=\eps.
\]
For a positive word $V$, let $|V|_H$ and $|V|_p$ denote the numbers of occurrences of $H$ and $p$ in $V$.

Starting from $(d_0,0)$, repeatedly read the first letter of the word in the current pair.  Every pair obtained in this way has the form $(d,i)$ with
\[
d\in\{0,1,H,p\}^+,
\qquad
i\in\{0,\ldots,m-1\}.
\]
Indeed, all outputs (and the starting word) are positive, and the word always contains exactly one occurrence of $p$.
 Put \(
(D_0,i_0)=(d_0,0).
\) If $(D_n,i_n)$ is defined, whenever we read a letter from $\{0,1\}$, we define $(D_{n+1},i_{n+1})$ to be the new pair carrying the discrepancy and the current state. Whenever we read $\#$, $H$, or $p$ we don't define a successor in this new sequence.

\begin{lemma} 
	For every $n$ for which $(D_n,i_n)$ is defined, we have that 
	\begin{equation*}
		i_n=n\bmod m,
		\qquad
		\beta(D_n)=W_n,
		\qquad
		|D_n|_H=n,
		\qquad \text{ and } \qquad
		|D_n|_p=1,
	\end{equation*}
	where $W_n$ is the queue of the cyclic tag system after $n$ steps.
	Moreover, $(D_{n+1},i_{n+1})$ is defined if and only if $W_n\neq\eps$.
\end{lemma}

\begin{proof}
	For $n=0$, we have $D_0=w_0p$ and $i_0=0$, so all assertions are immediate.
	
	Assume that the assertions hold for some $n$.  If $W_n=\eps$, then $\beta(D_n)=\eps$, so $D_n$ has no binary letter.  Therefore, we never read another binary letter and the successor is not defined.
	
	Suppose that $W_n\neq\eps$.  There is a unique factorization
	\begin{equation}\label{eq:Dn-factorization}
		D_n=ZaR,
	\end{equation}
	where
	\(
	Z\in\{H,p\}^*,\)
and
\(	a\in\{0,1\}.
	\)

	Read the letters of $Z$ in order.  The edges on $H$ and $p$ leave the state unchanged and output the letter that they read.  The resulting pair after reading $Z$ is thus
	\(
	(aRZ,i_n).
	\)
	
	If $a=0$, the next edge has label $0\mid H$ and gives
	\(
	(RZH,i_n+1).
	\)
	
	If $a=1$, the next edge has label $1\mid u_{i_n}H$ and gives
	\(
	(RZu_{i_n}H,i_n+1).
	\)

	Applying $\beta$ gives exactly the two cases of the cyclic tag rule.  In both cases the state increases by one, one new $H$ is added, and the unique occurrence of $p$ is preserved.
\end{proof}

\begin{proposition}\label{prop:first-letter-sequence}
	Consider the sequence of discrepancy pairs obtained from $(d_0,0)$ by repeatedly reading the first letter of the word in the current pair.
	\begin{enumerate}[label=\textup{(\roman*)}]
		\item If $\mathcal C$ does not halt, the sequence is infinite and no pair occurs twice.
		\item If $\mathcal C$ halts, the sequence reaches a pair $(D,i)$ with
		\(
		D\in\{H,p\}^+.
		\)
		Reading the letters of $D$ in order returns to the same pair $(D,i)$.
	\end{enumerate}
\end{proposition}

\begin{proof}
	Suppose first that $\mathcal C$ does not halt.  Then $W_n\neq\eps$ for every $n$, so  $(D_n,i_n)$ is defined for every $n$.
	
	After the $n$-th binary letter has been read and before the next binary letter is read, every word in the sequence contains exactly $n$ occurrences of $H$. Indeed, after the $n$-th binary letter is read, the sequence has value $(D_n,i_n)$ and reading $H$ or $p$ does not change the number of $H$'s in the discrepancy component. Therefore, if a pair appears twice in the sequence, both occurrences must be between the reading of the $n$-th and the $n+1$-th binary letters, for some $i$.
  Now, write $D_n=ZaR$ as in \eqref{eq:Dn-factorization}.  Before $a$ is read, the sequence moves the letters of $Z$ from the beginning to the end, one at a time.  At each step, the number of letters from $\{H,p\}$ before the first binary letter decreases by one.  Hence no word, and therefore no pair, occurs twice in this part of the sequence.
	
	Suppose now that $\mathcal C$ halts after $n$ steps.  Then $W_n=\eps$, so $D_n\in\{H,p\}^+$.  It is nonempty because it contains the unique occurrence of $p$.  Put $D=D_n$ and $i=i_n$.  Reading an initial $H$ or $p$ moves that letter to the end and leaves the state unchanged.  Reading all letters of $D$ therefore returns to $(D,i)$.
\end{proof}

\begin{remark}
	The letters $H$ and $p$ have different roles.  Every binary transition appends one $H$ to the discrepancy, and the number of occurrences of $H$ records the number of binary steps.  The letter $p$ occurs once in the initial word and is preserved by every transition used in the simulation.  It keeps the current  discrepancy nontrivial.
\end{remark}

Define a homomorphism
\(
\chi:F_A\longrightarrow\mathbb Z
\)
by
\[
\chi(p)=1,
\qquad
\chi(a)=0
\quad(a\in A\setminus\{p\}).
\]
So $\chi$ denotes the exponent-sum of $p$. For every edge
\(
q\xrightarrow{a\mid U}r
\)
whose endpoints are numbered states, one has
\begin{equation}\label{eq:chi-preserved}
	\chi(a)=\chi(U).
\end{equation}
Consequently, if this edge changes $(d,q)$ to $(e,r)$, then
\[
\chi(e)=-\chi(a)+\chi(d)+\chi(U)=\chi(d).
\]
Since $\chi(d_0)=1$, every pair reached from $(d_0,0)$ by a path among the numbered states has a nontrivial discrepancy.

Before proving the main theorem of this section,  we present two technical lemmas.
\begin{lemma}\label{lem:wrong-edge}
	Let
	\(
	i\xrightarrow{a\mid U}j
	\)
	be an edge between numbered states, and suppose it changes $(d,i)$ to $(e,j)$, where $\chi(d)=1$.  If
	\[
	a\neq\first(d),
	\]
	then
	\[
	\first(e)=a^{-1}.
	\]
\end{lemma}

\begin{proof}
	The word $a^{-1}d$ is freely reduced and begins with $a^{-1}$. It is easy to see that multiplication by $U$ on the right cannot cancel this first letter. Indeed, $a^{-1}d$ contains a $p$ and if $a\not\in\{p,p^{-1}\}$, then $U$ contains no $p^{-1}$ to cancel that $p$ and the initial $a^{-1}$ survives cancellation. If $a\in\{p,p^{-1}\}$, then $U=a$ and so $U$ has length one and it cannot cancel the initial $a^{-1}$ as $d$ is nontrivial.
\end{proof}

The next lemma  concerns an arbitrary closed path, not necessarily the path obtained by repeatedly reading the first letter.

\begin{lemma}\label{lem:closed-path-at-zero}
	Let $\rho$ be a nonempty reduced closed path at state $0$ that uses only the numbered states.  If reading $\rho$ from $(d_0,0)$ returns to $(d_0,0)$, then the sequence obtained from $(d_0,0)$ by repeatedly reading the first letter contains the same pair at two different times.
\end{lemma}

\begin{proof}
	Write
	\(
	\rho=e_1\cdots e_k,
	\)
	and let
	\[
	(d_0,0),(d_1,i_1),\ldots,(d_{k-1},i_{k-1}),	(d_0,0)
	\]
	be the successive pairs obtained while reading $\rho$ from $(d_0,0)$.  
	By \eqref{eq:chi-preserved},
	\[
	\chi(d_j)=1
	\qquad(0\leq j\leq k).
	\]
	In particular, every $d_j$ is nontrivial.
	
	For $1\leq j\leq k$, let $a_j$ be the input label of $e_j$.  We say that the edge $e_j$ is positive if
		\[
	a_j=\first(d_{j-1}),
	\]
	and negative otherwise.  If $e_j$ is negative, Lemma~\ref{lem:wrong-edge} gives
	\[
	\first(d_j)=a_j^{-1}.
	\]
	Thus the reverse edge $e_j^{-1}$ reads the first letter from $(d_j,i_j)$.
	
	A negative edge cannot be followed by a positive one.  Indeed, if $e_j$ is negative and $e_{j+1}$ is positive, then  both $e_j^{-1}$ and $e_{j+1}$ have as input $\first(d_j)$.  Determinism would give
	\(
	e_{j+1}=e_j^{-1},
	\)
	which contradicts the fact that $\rho$ is reduced.
	
	It follows that, for some $r\in\{0,\ldots,k\}$, the first $r$ edges are positive and the remaining are negative.
	
	If $r=k$, then $\rho$ is an initial part of the sequence obtained by repeatedly reading the first letter and returns to $(d_0,0)$.  If $r=0$, then $\rho^{-1}$ has this property.  In either case, the sequence repeats $(d_0,0)$.

	Assume that $0<r<k$.  The path
	\(
	e_1\cdots e_r
	\)
	is an initial part of the  sequence obtained by repeatedly reading the first letter from $(d_0,0)$ to $(d_r,i_r)$.  The path
	\(
	e_k^{-1}\cdots e_{r+1}^{-1}
	\)
	is another initial part of the same sequence with the same endpoints.  If these paths have different lengths, the sequence reaches $(d_r,i_r)$ at two different times.  If they have the same length, determinism implies that they are equal edge by edge.  Their last edges are then equal, so
	\(
	e_r=e_{r+1}^{-1},
	\)
	again contradicting the reducedness of $\rho$.  Thus the   sequence repeats a pair.
\end{proof}

\begin{theorem}\label{thm:partial trans}
	From a cyclic tag system $\mathcal C$, one can effectively construct a finite partial deterministic inverse transducer $\mathcal T_{\mathcal C}$ such that the following conditions are equivalent:
	\begin{enumerate}[label=\textup{(\roman*)}]
		\item $\mathcal C$ halts;
		\item $\mathcal T_{\mathcal C}$ has a nontrivial fixed loop at $s$;
		\item $\Fix(\widetilde{\mathcal T}_{\mathcal C})\neq\{1\}$.
	\end{enumerate}
	Moreover, if $g\in\Fix(\widetilde{\mathcal T}_{\mathcal C})$, then the freely reduced representative of $g$ labels a closed path at $s$.
\end{theorem}

\begin{proof}
	We first prove the last assertion.  Let $g$ be fixed, and let $w$ be its freely reduced representative.  Since $w$ is freely reduced, its path  is reduced.  If $w$ is empty, it is already a closed path at $s$. 
	Suppose that $w$ is nonempty.  The only edge leaving $s$ is $e_{\#}$, so the path begins with that edge and the discrepancy pair is $(d_0,0)$.  If the path ended at a numbered state, its final first coordinate would have $p$-exponent sum one by \eqref{eq:chi-preserved}.  This is impossible because $g$ is fixed, so it has trivial discrepancy.
	Therefore the path returns to $s$ through $e_{\#}^{-1}$.  It cannot leave $s$ again, because the only possible next edge would be $e_{\#}$, giving an immediate reversal.  Hence the path ends at $s$.
	
	It follows immediately that conditions \textup{(ii)} and \textup{(iii)} are equivalent: a fixed loop is a fixed point, and every fixed point is represented by a loop at $s$.
	
	Suppose that $\mathcal C$ halts.  By Proposition~\ref{prop:first-letter-sequence}, the  sequence obtained by repeatedly reading the first letter starting from $(d_0,0)$ reaches a pair $(D,i)$ with $D\in\{H,p\}^+$, and reading the letters of $D$ returns to $(D,i)$.  Let $\tau$ be the path from $(d_0,0)$ to $(D,i)$, and let $\sigma$ be the closed path at state $i$ whose input (and output) word is $D$.  
	Therefore
	\[
	\begin{aligned}
		(1,s)
		\xrightarrow{e_{\#}}(d_0,0)
		\xrightarrow{\tau}(D,i)
		\xrightarrow{\sigma}(D,i)
		\xrightarrow{\tau^{-1}}(d_0,0)
		\xrightarrow{e_{\#}^{-1}}(1,s).
	\end{aligned}
	\]
	is a closed path at $s$ which sends $(1,s)$ to itself.
 It is therefore a fixed loop.  Its input label represents a conjugate of the nonempty positive word $D$, and is nontrivial.  This proves \textup{(i)}$\Rightarrow$\textup{(ii)}.
	
	Conversely, suppose that $\mathcal C$ does not halt and that there is a nontrivial fixed reduced loop at $s$.       Since $s$ is incident with only the geometric edge $e_{\#}$, the reduced loop has the form
	\(
	e_{\#}\rho e_{\#}^{-1},
	\)
	where $\rho$ is a nonempty reduced closed path at state $0$ using only the numbered states.
	
	Suppose that reading $\rho$ from $(d_0,0)$ gives $(e,0)$.  The inverse entry edge has label
	\[
	\#^{-1}\mid d_0^{-1}\#^{-1}.
	\]
	Since the full path is fixed,  \eqref{eq:update} gives
	\[
	1=\red{\# e d_0^{-1}\#^{-1}},
	\]
	and therefore $e=d_0$.  Thus $\rho$ sends $(d_0,0)$ to itself.  Lemma~\ref{lem:closed-path-at-zero} implies that the first-letter sequence repeats a pair, and so $\mc C$ halts by Proposition~\ref{prop:first-letter-sequence}(i), which is absurd.  
\end{proof}

Since the halting problem is undecidable for cyclic tag systems, we have the following corollary.
\begin{corollary} 
	It is undecidable whether a finite partial deterministic inverse transducer has a nontrivial fixed point.   
\end{corollary}

\section{The Post Correspondence Problem for free groups}

The purpose of this section is to show that the problem of deciding the existence of a nontrivial fixed loop in a finite partial deterministic inverse transducer can be reduced to an instance of the PCP for free groups. This, together with the main result from the previous section, leads to the undecidability of the PCP for free groups.

\begin{theorem} \label{thm:transducer-to-equalizer}
	From a finite partial inverse $\wt A$-transducer $\mathcal T=(Q,q_0,\delta, \lambda)$, one can effectively construct a finite set $Y$, and homomorphisms
	\[
	g,h:F_Y\longrightarrow F_A
	\]
	between two free groups
	such that $h$ is injective and
	\[
	\Eq(g,h)\neq\{1\}
	\quad\Longleftrightarrow\quad
	\mathcal T\text{ has a nontrivial fixed loop at }q_0.
	\]
\end{theorem}

 \begin{proof}
 	Forget the output labels of $\mathcal T$ and consider only its
 	input-labelled automaton.  Let
 	\[
 	K=
 	\{w\in F_A\mid (q_0,w)\widetilde\delta=q_0\}
 	\]
 	be the set of elements of $F_A$ which label closed paths
 	at $q_0$.
 	
 	The set $K$ is a subgroup of $F_A$.  Indeed, the product of two
 	closed paths is a closed path, and the inverse of a closed path is
 	again a closed path.
 	
 	Consider the connected component of the input automaton containing $q_0$. Since it is   finite, deterministic, and inverse, it
 	is a finite Stallings automaton for $K$, after removing any edges
 	which don't occur in any reduced closed path at $q_0$.

 	 Therefore one can
 	effectively compute a finite free basis $B$
 	of $K$. Let
 	\(
 	B=\{b_1,\ldots,b_r\}
 	\)
 	be such a basis, and let
 	\[
 	Y=\{y_1,\ldots,y_r\}
 	\]
 	be a new alphabet.

 	Define $\psi=\wt \T|_K$. That is, 
 	\(
 	\psi:K\longrightarrow F_A
 	\)
 	is defined by
 	\(
 	w\psi=w\widetilde{\mathcal T}.
 	\)
 	This is a homomorphism.  Indeed, if $u,v\in K$, then reading $u$
 	returns the transducer to $q_0$.  Hence
 	\[
 	(uv)\psi
 	=
 	(uv)\widetilde{\mathcal T}
 	=
 	(u\widetilde{\mathcal T})
 	(v\widetilde{\mathcal T})
 	=
 	(u\psi)(v\psi).
 	\]
 	
 	Let $F_Y$ be the free group with basis $Y$.  Define homomorphisms
 	\(
 	g,h:F_Y\longrightarrow F_A
 	\)
 	 by
 	\[
 	y_ih=b_i
 	\qquad\text{and}\qquad
 	y_ig=b_i\psi \qquad (1\leq i\leq r).
 	\]
 	
 	Since $B$ is a free basis of $K$, the homomorphism $h$ is injective
 	and
 	\(
 	F_Yh=K.
 	\)
 	Moreover, by construction,
 	\begin{equation}\label{eq:g-psi-h}
 		xg=(xh)\psi
 	\end{equation}
 	for every $x\in F_Y$.
 	
 	Suppose first that there is some nontrivial $x$ in 
 $\Eq(g,h).$
 	Put
 	\(
 	w=xh.
 	\)
 	Since $h$ is injective, $w\neq1$ and, since $w\in K$,  $w$ labels a
 	closed path at $q_0$.  By \eqref{eq:g-psi-h},
 	\[
 	w\widetilde{\mathcal T}
 	=
 	w\psi
 	=
 	xg
 	=
 	xh
 	=
 	w.
 	\]
 	Thus $w$ is a nontrivial fixed loop at $q_0$.
 	
 	Conversely, suppose that $w$ is a nontrivial fixed loop at $q_0$.
 	Then $w\in K$ and
 	\(
 	w\psi=w.
 	\)
 	Since $h$ maps $F_Y$ isomorphically onto $K$, there exists
 	$x\in F_Y$ such that
 	\(
 	xh=w.
 	\)
 	The element $x$ is nontrivial because $w\neq1$ and $h$ is
 	injective.  Using \eqref{eq:g-psi-h}, we obtain
 	\(
 	xg
 	=
 	(xh)\psi
 	=
 	w\psi
 	=
 	w
 	=
 	xh.
 	\)
 	Hence
 	\[
 	1\neq x\in\Eq(g,h).
 	\]
 	
 	Finally, the construction is effective.  From the finite input
 	automaton one can compute a free basis $B$ of $K$.  For each
 	$b_i\in B$, one can read the corresponding closed path from $q_0$
 	and compute its output label.  These words are precisely the images
 	of the generators under $h$ and $g$.
 \end{proof}
 
 \begin{corollary}\label{pcp undec}
  The PCP is undecidable for free groups.
 \end{corollary}

\section{Some consequences of the undecidability of PCP}
In this section, we derive several consequences of the undecidability of the Post Correspondence Problem. We start with immediate corollaries, and then focus on the study of fixed points of virtual endomorphisms and algorithmic problems in certain matrix groups.

\begin{corollary} 
	There is no algorithm which, given homomorphisms
	\(
	g,h:F_\Sigma\to F_\Delta
	\)
	with at least one of \(g,h\) injective, outputs a basis for
	\(
	\Eq(g,h).
	\)
	
\end{corollary}

\begin{corollary} 
	There is no algorithm which, given homomorphisms
	\(
	g,h:F_\Sigma\to F_\Delta
	\)
	with at least one of \(g,h\) injective, outputs the rank of
	\(
	\Eq(g,h).
	\)
\end{corollary}

Since every free group embeds in $F_2$ (and in any free group of rank at least 2), we observe that the PCP is  still undecidable when we restrict ourselves to homomorphisms with $F_2$ as the codomain or to endomorphisms.

\begin{corollary} \label{cor:target-F2}
	The Post Correspondence Problem is undecidable for pairs
	\[
	g,h:F_Y\longrightarrow F_2,
	\]
	even when $h$ is injective.
\end{corollary}

\begin{proof}
	Let $g,h:F_Y\to F_A$ be two homomorphisms between free groups, with $h$ injective. Let
	\(
	\iota:F_A\longrightarrow F_2
	\)
	be  an injective homomorphism.  Define
	\[
	g'= g\iota,
	\qquad
	h'= h\iota.
	\]
	Injectivity of $\iota$ gives
	\[
	\Eq(g',h')=\Eq(g,h),
	\]
	and $h'$ is injective.   
\end{proof}

\begin{corollary} 
	There is no algorithm which, given two endomorphisms
	\[
	\alpha,\beta:F_Y\longrightarrow F_Y
	\]
	of a finitely generated free group, with $\beta$ injective, decides whether $\Eq(\alpha,\beta)$ is nontrivial.   
\end{corollary}

\begin{proof}
	Let $g,h:F_Y\to F_A$ be two homomorphisms between free groups, with $h$ injective. Let
	\(
	\iota:F_A\longrightarrow F_Y
	\)
	be  an injective homomorphism. As above, defining
	\[
	g'= g\iota,
	\qquad
	h'= h\iota,
	\]
	we have that 
	\[
	\Eq(g',h')=\Eq(g,h),
	\]
	and $h'$ is injective.   
\end{proof}

The following corollary follows directly from \cite[Theorem 4.1]{[Car26]} and Corollary \ref{pcp undec}.
\begin{corollary}
	There is no algorithm that takes as input two free groups $F_n$ and $F_m$, and two endomorphisms $\phi,\psi\in \End(F_n\times F_m)$ that decides whether $\Fix(\phi)\cap \Fix(\psi)$ is trivial or not.
\end{corollary}
We then have that there is an algorithm taking as input a monomorphism and an endomorphism of $F_n\times F_m$ that decides if the intersection of their fixed subgroups is trivial or not \cite[Corollary 3.4]{[Car26]}, but when we remove the injectivity assumption, the problem becomes undecidable.

\subsection{Complete transducers and virtual endomorphisms}
 A homomorphism $\phi:H\to G$ where $H$ is a finite index subgroup of $G$ is said to be a \emph{virtual endomorphism}. We now prove that, although the fixed subgroup of an endomorphism of the free group is computable, when considering virtual endomorphisms, triviality of the fixed subgroup cannot be decided.

\begin{theorem}\label{thm:virtual-triviality}
	There is no algorithm which, given a finitely generated free group $F$, a finite-index
	subgroup $H\leq F$, and a homomorphism
	\(
	\theta:H\longrightarrow F,
	\)
	decides whether
	\[
	\mathrm{Fix}(\theta)=\{1\}.
	\]
\end{theorem}

\begin{proof}
	Let $\mathcal C$ be a cyclic tag system and let $\T_{\mathcal C}$ be the partial inverse
	transducer constructed in Section~3. Put
	\[
	B=\{0,1,H,p\},
	\qquad
	d_0=w_0p,
	\]
	and let $\mc R_{\mathcal C}$ be the transducer obtained from $\T_{\mathcal C}$ by deleting the
	state $s$ and the  $\#$-edge. Thus $\mc R_{\mathcal C}$ has state set
	\(
	Q=\{0,\ldots,m-1\}
	\)
	and is a finite complete deterministic inverse transducer over $\widetilde B$.
	
	Define
	\[
	H_{\mathcal C}
	=
	\{x\in F_B\mid(0,x)\widetilde\delta=0\}
	\]
	to be the input labels of reduced paths labelling loops at $0$.
	Since $\mc R_{\mathcal C}$ is complete and inverse, $F_B$ acts on the finite set $Q$, and
	$H_{\mathcal C}$ is the stabilizer of $0$. Hence $H_{\mathcal C}\leq F_B$ has finite index.
	
	For $x\in H_{\mathcal C}$, define
	\[
	x\psi_{\mathcal C}=x\widetilde{\mc R}_{\mathcal C}.
	\]
	The map $\psi_{\mathcal C}:H_{\mathcal C}\longrightarrow F_B$ is a homomorphism,
	since every element of $H_{\mathcal C}$ returns the transducer to state $0$.
	
	Define
	\(
	\theta_{\mathcal C}:H_{\mathcal C}\longrightarrow F_B
	\)
	by
	\[
	x\theta_{\mathcal C}=d_0(x\psi_{\mathcal C})d_0^{-1}.
	\]
	Let $x\in H_{\mathcal C}$. The path in $\T_{\mathcal C}$ labelled by
	\[
	\#x\#^{-1}
	\]
	has output
	\[
	\#d_0(x\psi_{\mathcal C})d_0^{-1}\#^{-1}.
	\]
	Consequently,
	\[
	\#x\#^{-1}
	\text{ is fixed}
	\quad\Longleftrightarrow\quad
	x=x\theta_{\mathcal C}.
	\]
 Therefore
	\[
	\Fix(\theta_{\mathcal C})\neq\{1\}
	\quad\Longleftrightarrow\quad
	\T_{\mathcal C}\text{ has a nontrivial fixed loop at }s.
	\]
	By Theorem~\ref{thm:partial trans}, this holds if and only if $\mathcal C$ halts.
\end{proof}

Fixed subgroups of virtual endomorphisms are closely related to fixed loops of complete inverse transducers. Silva proved in \cite[Theorem 3.2]{[Sil13]} that  for every finite complete inverse transducer
\(\mathcal T\), the set $\Fix(\wt\T)$ was a rational subset of the free group, which, from Benois's Theorem means only that the language of all reduced words representing elements of $\Fix(\wt\T)$ is rational \cite{[Ben79]}.
 The next result shows that Theorem \ref{thm:virtual-triviality} implies that 
this regular language cannot be constructed effectively from
\(\mathcal T\). In fact, it cannot be decided whether such a transducer has a nontrivial fixed point.  We start with a technical lemma.

\begin{lemma}\label{lem:state-conjugation}
	Let $\T=(Q,q_0,\delta,\lambda)$ be a finite complete deterministic inverse
	$\widetilde A$-transducer, and choose elements $c_q\in F_A$ for $q\in Q$. Construct a transducer $\T^c$ with the same states and input
	transitions by replacing every edge
	\[
	q\xrightarrow{\ a\mid U\ }r
	\]
	with
	\[
	q\xrightarrow{\ a\mid c_q^{-1}Uc_r\ }r.
	\]
	Then $\T^c$ is a finite complete deterministic inverse transducer. Moreover, if reading
	$x\in F_A$ from $q_0$ terminates at $q$, then
	\begin{equation}\label{eq:state-conjugation}
		x\widetilde{\T^c}=c_{q_0}^{-1}(x\widetilde \T)c_q.
	\end{equation}
\end{lemma}

\begin{proof}
	The input of each transition is unchanged, so $\T^c$ is finite, complete, and deterministic. Also $\T^c$ is inverse: if $\T$ has the transitions
	\[
	q\xrightarrow{\ a\mid U\ }r\qquad \text{ and } \qquad 
	r\xrightarrow{\ a^{-1}\mid U^{-1}\ }q,
	\]
	we transform them into the following transitions:
	\[
	q\xrightarrow{\ a\mid c_q^{-1}Uc_r\ }r\qquad \text{ and } \qquad 
	r\xrightarrow{\ a^{-1}\mid c_r^{-1}U^{-1}c_q\ }q.
	\]
	These are mutually inverse, so $\T^c$ is inverse.
	Formula~\eqref{eq:state-conjugation} can be easily seen by induction on the length, performing the cancellations along the way.
\end{proof}

\begin{theorem}\label{thm:complete-fixed-point}
	From a cyclic tag system $\mathcal C$, one can effectively construct a finite complete
	deterministic inverse transducer ${\mc S}_{\mathcal C}$ such that the following conditions are
	equivalent:
	\begin{enumerate}[label=(\roman*)]
		\item $\mathcal C$ halts;
		\item ${\mc S}_{\mathcal C}$ has a nontrivial fixed loop;		\item $\Fix(\widetilde {\mc S}_{\mathcal C})\neq\{1\}$.
	\end{enumerate}
	Moreover, every fixed point of $\widetilde S_{\mathcal C}$ is represented by a closed path.
\end{theorem}

\begin{proof}
	Retain the notation
	\[
	B=\{0,1,H,p\},
	\qquad
	d_0=w_0p,
	\]
	and let $\mc R_{\mathcal C}$ be the complete numbered-state transducer used in the proof of
	Theorem~\ref{thm:virtual-triviality}. Define
	\(
	\chi:F_B\longrightarrow\mathbb Z
	\)
	by
	\[
	\chi(p)=1,
	\qquad
	\chi(0)=\chi(1)=\chi(H)=0.
	\]
	That is $\chi$ is the exponent-sum of $p$.
	As in \eqref{eq:chi-preserved}, every transition
	\(
	i\xrightarrow{\ a\mid U\ }j
	\)
	of $\mc R_{\mathcal C}$ satisfies
	\(
	\chi(U)=\chi(a).
	\)
	
	For $i\in\{0,\ldots,m-1\}$, put
	\(
	c_i=d_0^{-1}p^i.
	\)
	Thus
	\[
	c_0=d_0^{-1},
	\qquad \text{ and } \qquad
	\chi(c_i)=i-1.
	\]
	Let
	\(
	\mc S_{\mathcal C}=\mc R_{\mathcal C}^{c}
	\)
	be the transducer obtained as in Lemma~\ref{lem:state-conjugation}.
	
	Let $x\in F_B$, and suppose that reading $x$ from state $0$ terminates at state $j$.
	Formula~\eqref{eq:state-conjugation} gives
	\[
	x\widetilde{\mc S}_{\mathcal C}
	=
	c_0^{-1}(x\widetilde{\mc R}_{\mathcal C})c_j.
	\]
	Since $\chi$ agrees on the input and output of every transition of $\mc R_{\mathcal C}$,
	\[
	\chi(x\widetilde{\mc R}_{\mathcal C})=\chi(x).
	\]
	Consequently,
	\[
	\chi(x\widetilde{\mc S}_{\mathcal C})-\chi(x)
	=-\chi(c_0)+\chi(c_j)=j.
	\]
	If $x\widetilde{\mc S}_{\mathcal C}=x$, then $j=0$. Hence every fixed point of
	$\widetilde S_{\mathcal C}$ labels a closed path at state $0$, and conditions (ii) and (iii)
	are equivalent.
	
	Now let $H_{\mc C}$ and $\psi_{\mc C}$ and $\theta_{\mc C}$ be the subgroup and homomorphism from the proof of Theorem ~\ref{thm:virtual-triviality}. Fixed points of $\theta_{\mc C}$ must belong to $H_{\mc C}$ by definition, so they label loops at $0$.
 	Therefore, for $x\in H_{\mc C}$,
	\[
	x\widetilde{\mc S}_{\mathcal C}
	=
		c_0^{-1}(x\widetilde{\mc R}_{\mathcal C})c_0
	=
	d_0(x\psi_{\mathcal C})d_0^{-1}
	=
	x\theta_{\mathcal C}.
	\]
	Thus
	\(	\Fix(\widetilde{\mc S}_{\mathcal C})=\Fix(\theta_{\mathcal C}).
	\) and the result follows from the proof of Theorem~\ref{thm:virtual-triviality}.
\end{proof}

\begin{corollary}
	It is undecidable whether a finite complete deterministic inverse transducer has a nontrivial
	fixed point.
\end{corollary}

This construction yields undecidability of PCP even when one of the homomorphisms is injective and has a finite index image.
\begin{corollary} 
	The PCP is undecidable for free groups, even when one of the two homomorphisms is
	injective and has finite-index image.
\end{corollary}

\begin{proof}
	Apply Theorem~\ref{thm:transducer-to-equalizer} to the complete transducers
	$\mc S_{\mathcal C}$ from Theorem~\ref{thm:complete-fixed-point}. Since every fixed point of
	$\mc S_{\mathcal C}$ is a fixed loop, the corresponding equalizer is isomorphic to
	$\Fix(\widetilde{\mc S}_{\mathcal C})$. Since $\mc S_{\mathcal C}$ is complete, the image of the
	injective homomorphism supplied by Theorem~\ref{thm:transducer-to-equalizer} has finite
	index.
\end{proof}

\section{Applications to groups of matrices over the integers}

We conclude with two consequences for algorithmic problems in arithmetic
matrix groups.

Let \(R\) be a subring of \(\mathbb C\), let
\(
M_1,\ldots,M_k\in \operatorname{GL}_n(R),
\)
and put
\(
G=\langle M_1,\ldots,M_k\rangle.
\)
The \emph{stabilizer problem} asks, given in addition a vector
\(
u\in R^n,
\)
whether
\(
\operatorname{Stab}_G(u)
=
\{M\in G\mid Mu=u\}
\)
is nontrivial.

The \emph{upper-right-corner problem} asks whether there exists a
nonidentity matrix
\(
I_n\neq M\in G
\)
such that
\(
M_{1,n}=0.
\)

Breuillard and Kocharyan proved that the stabilizer problem is
undecidable for finitely generated subgroups of
\(\operatorname{GL}_n(\mathbb Z)\) in dimension at least nine, and that
the upper-right-corner problem is undecidable for finitely generated
subgroups of \(\operatorname{GL}_n(\mathbb Q)\) in dimension at least
nine \cite{[BK25]}. Their matrix-group reductions use the undecidable word
problem and Mihailova-type membership constructions.

Using the Post Correspondence Problem for free groups, we obtain
undecidability in dimensions four and five, respectively. Our results
hold for integral matrices of determinant one, under the additional
promise that the matrices in the input tuple freely generate the subgroup
which they generate. For the stabilizer problem, the vector is fixed in
advance and is not part of the input.

The classical Post Correspondence Problem has long been used to prove
undecidability results for matrix semigroups. Breuillard and Kocharyan
recently considered the analogous questions for matrix groups \cite{[BK25]} following work by Dixon \cite{[Dix85]}. They proved that
the stabilizer problem is undecidable in dimension at least nine; although
their theorem is stated over \(\mathbb Q\), their proof already gives the
corresponding result for finitely generated subgroups of
\(\operatorname{GL}_n(\mathbb Z)\) acting on \(\mathbb Z^n\). They also proved
that the upper-right-corner problem is undecidable for finitely generated
subgroups of \(\operatorname{GL}_n(\mathbb Q)\) in dimension at least nine
\cite{[BK25]}. 
 We will use
 Corollary~\ref{cor:target-F2} to lower the dimension (a stronger version) of the integral
 stabilizer problem to four and to prove (a stronger version of) the upper-right-corner result over
 the integers in dimension five. In particular, we answer the question posed by Breuillard and Kocharyan of whether the upper-right-corner problem was undecidable for integral matrices.

Fix an effective embedding
\[
\rho:F_2\longrightarrow\operatorname{SL}_2(\mathbb Z)
\]
such that every nonidentity element of \(\rho(F_2)\) is hyperbolic. Such an
embedding is given, for example, by mapping a free basis of \(F_2\) to
\[
\begin{pmatrix}3&2\\1&1\end{pmatrix}
\qquad\text{and}\qquad
\begin{pmatrix}1&1\\2&3\end{pmatrix}.
\]
Indeed, these matrices freely generate a subgroup all of whose nonidentity elements
have trace of absolute value greater than two
\cite[Theorem~6]{[BK25]}.

We start by showing an undecidability result concerning the intersection of free matrix subgroups with a fixed free subgroup of $\operatorname{SL}_4(\Z)$ which, to the authors knowledge, appears to be new.

Let
\[
\Delta_\rho
=
\left\{
\begin{pmatrix}
	\rho(a)&0\\
	0&\rho(a)
\end{pmatrix}
:a\in F_2
\right\}
\leq\operatorname{SL}_4(\mathbb Z).
\]
Thus $\Delta_\rho$ is a   free subgroup of rank two.

\begin{proposition}
	There is no algorithm which, given matrices promised to be a free basis of
	a free subgroup
	\(
	K\leq\operatorname{SL}_4(\mathbb Z),
	\)
	decides whether
	\[
	K\cap\Delta_\rho\neq\{1\}.
	\]
\end{proposition}

\begin{proof}
	Let
	\[
	g,h:F_Y\longrightarrow F_2
	\]
	be an instance from Corollary~\ref{cor:target-F2}, with $h$ injective. For
	$x\in F_Y$, put
	\[
	x\eta
	=
	\begin{pmatrix}
		xg\rho&0\\
		0&xh\rho
	\end{pmatrix}
	\]
	and let $K=F_Y\eta$. Since $h$ and $\rho$ are injective, so is $\eta$.
	Hence, if $Y=\{y_1,\ldots,y_r\}$, then the matrices
	\(
	y_1\eta,\ldots,y_r\eta
	\)
	form a free basis of $K$.
	
	Moreover,
	\[
	K\cap\Delta_\rho
	=
	\{x\eta:x\in\operatorname{Eq}(g,h)\}.
	\]
	Indeed, $x\eta$ lies in $\Delta_\rho$ if and only if
	$xg\rho=xh\rho$, which, by injectivity of $\rho$, is equivalent to
	$xg=xh$. The result follows from Corollary~\ref{cor:target-F2}.
\end{proof}

Let
\(
V=M_2(\mathbb Z),
\)
viewed as a free abelian group of rank four. The group
\(
\operatorname{SL}_2(\mathbb Z)\times
\operatorname{SL}_2(\mathbb Z)
\)
acts on \(V\) by
\begin{equation}\label{eq:left-right-action}
	(A,B)\cdot X=AXB^{-1}.
\end{equation}
After choosing an integral basis of \(V\), the action of $(A,B)$ on $V$  is represented by
matrices in \(\operatorname{SL}_4(\mathbb Z)\).

\begin{lemma}\label{lem:matrix-diagonal-stabilizer}
	The action~\eqref{eq:left-right-action} has the following properties.
	\begin{enumerate}
		\item Every \((A,B)\in
		\operatorname{SL}_2(\mathbb Z)\times\operatorname{SL}_2(\mathbb Z)\)
		acts on \(V\) with determinant one.
		
		\item The restriction of the action to
		\(
		\rho(F_2)\times\rho(F_2)
		\)
		is faithful.
		
		\item For \(A,B\in\rho(F_2)\), one has
		\[
		(A,B)\cdot I_2=I_2
		\quad\Longleftrightarrow\quad
		A=B.
		\]
	\end{enumerate}
\end{lemma}

\begin{proof}
	The action law follows from
	\begin{align*}
		(A_1,B_1)\cdot\bigl((A_2,B_2)\cdot X\bigr)
		&=
		A_1(A_2XB_2^{-1})B_1^{-1}\\
		&=
		(A_1A_2)X(B_1B_2)^{-1}\\
		&=
		(A_1A_2,B_1B_2)\cdot X.
	\end{align*}
	
	Write \(L_A(X)=AX\) and \(R_{B^{-1}}(X)=XB^{-1}\). Under the
	identification of \(M_2(\mathbb Z)\) with the direct sum of its two
	columns, \(L_A\) is represented by the block diagonal matrix
	\[
	\begin{pmatrix}
		A&0\\
		0&A
	\end{pmatrix},
	\]
	and therefore
	\(
	\det L_A=(\det A)^2.
	\)
	Similarly, 
	\(
	\det R_{B^{-1}}=(\det B^{-1})^2,
	\)
	and so,
	\[
	\det(X\mapsto AXB^{-1})
	=
	(\det A)^2(\det B^{-1})^2
	=
	1.
	\]
	Suppose that \(A,B\in\rho(F_2)\) and that \((A,B)\) acts trivially on
	\(V\). Applying it to \(I_2\) gives
	\[
	AB^{-1}=I_2,
	\]
	and hence \(A=B\). It follows that
	\(
	AXA^{-1}=X
	\)
	for every \(X\in M_2(\mathbb Z)\). Therefore \(A\) is scalar, so
	\(
	A=B=\pm I_2.
	\)
	Since \(\rho(F_2)\) is torsion-free, it does not contain \(-I_2\). Hence
	\(
	A=B=I_2,
	\)
	which proves faithfulness.
	
	Finally,
	\[
	(A,B)\cdot I_2=AB^{-1},
	\]
	and this is equal to \(I_2\) if and only if \(A=B\).
\end{proof}

We first apply this action to the stabilizer problem. We remark that our construction allows to prove the undecidability of the stabilizer problem in the stronger version where our subgroup is free, our vector being stabilized is the identity and the matrices have determinant $\pm 1$.

\begin{theorem}\label{thm:matrix-stabilizer}
	There is no algorithm which, given a finite tuple of matrices promised to
	be a free basis of a subgroup
	\(
	G\leq\operatorname{SL}_4(\mathbb Z),
	\)
	decides whether
	\[
	\operatorname{Stab}_G(e_1)
	=
	\{M\in G:Me_1=e_1\}
	\]
	is nontrivial.
\end{theorem}

\begin{proof}
	Let
	\(
	g,h:F_Y\longrightarrow F_2
	\)
	be an instance from Corollary~\ref{cor:target-F2}, with \(h\) injective.
	
	Let \(E_{ij}\) denote the standard matrix units and choose the integral
	basis
	\[
	\mathcal B=(I_2,E_{12},E_{21},E_{22})
	\]
	of \(V\). With respect to this basis, the coordinate vector of \(I_2\) is
	\(e_1\).
	
	For \(A,B\in\operatorname{SL}_2(\mathbb Z)\), let
	\(
	\Lambda_{\mathcal B}(A,B)\in\operatorname{SL}_4(\mathbb Z)
	\)
	denote the matrix, with respect to \(\mathcal B\), of the linear
	transformation
	\[
	X\longmapsto (A,B)\cdot X=AXB^{-1}.
	\]
	For every \(y\in Y\), put
	\[
	M_y=
	\Lambda_{\mathcal B}\bigl(yg\rho,yh\rho\bigr).
	\]
	Let
	\(
	G=\langle M_y:y\in Y\rangle
	\)
	
	and 
	\(
	\eta:F_Y\longrightarrow G
	\)
	be the projection given by
	\[
	x\eta=
	\Lambda_{\mathcal B}\bigl(xg\rho,xh\rho\bigr).
	\]
	Since \(h\) and \(\rho\) are injective, and, by 	Lemma~\ref{lem:matrix-diagonal-stabilizer}, the action is faithful on
	\(\rho(F_2)\times\rho(F_2)\), then \(\eta\) is injective, and 
	$\{M_y\mid y\in Y\}$
	is a free basis of \(G\).
	
	For every \(x\in F_Y\), the coordinate vector \(e_1\) represents \(I_2\),
	and therefore Lemma~\ref{lem:matrix-diagonal-stabilizer} gives
	\begin{align*}
		(x\eta)e_1=e_1
		&\Longleftrightarrow
		\bigl(xg\rho,xh\rho\bigr)\cdot I_2=I_2\\
		&\Longleftrightarrow
		xg\rho=xh\rho\\
		&\Longleftrightarrow
		xg=xh.
	\end{align*}
	Since \(\eta\) is injective, \(G\) contains a nonidentity element fixing
	\(e_1\) if and only if \(\Eq(g,h)\) contains a nontrivial element. The
	result follows from Corollary~\ref{cor:target-F2}.
\end{proof}

\begin{corollary}
	There is a fixed nonzero square-zero matrix $N\in M_5(\mathbb Z)$ such that
	there is no algorithm which, given a free basis of a free subgroup
	\[
	G\leq\operatorname{SL}_5(\mathbb Z),
	\]
	decides whether $G$ contains a nonidentity matrix commuting with $N$.
\end{corollary}

\begin{proof}
	Let
	\[
	N=
	\begin{pmatrix}
		0& 0 & 0& 0 &1\\
				0& 0 & 0& 0 &0\\
								0& 0 & 0& 0 &0\\
												0& 0 & 0& 0 &0\\
																0& 0 & 0& 0 &0\\

	\end{pmatrix}.
	\]
	Then $N\neq0$ and $N^2=0$. Embed
	$\operatorname{SL}_4(\mathbb Z)$ in $\operatorname{SL}_5(\mathbb Z)$ by
	$M\mapsto\operatorname{diag}(M,1)$. A direct computation gives
	\[
	\operatorname{diag}(M,1)N=N\operatorname{diag}(M,1)
	\quad\Longleftrightarrow\quad
	Me_1=e_1.
	\]
	The result follows from Theorem~\ref{thm:matrix-stabilizer}.
\end{proof}

We now prove the undecidability of the upper-right problem for integer matrices, again in the stronger case where the subgroup is promised to be free and the matrices have determinant $\pm 1$.

\begin{theorem}\label{thm:matrix-corner}
	There is no algorithm which, given a finite tuple of matrices promised to
	be a free basis of a subgroup
	\(
	G\leq\operatorname{SL}_5(\mathbb Z),
	\)
	decides whether \(G\) contains a nonidentity matrix whose
	\((1,5)\)-entry is zero.
\end{theorem}

\begin{proof}
	Let
	\(
	g,h:F_Y\longrightarrow F_2
	\)
	be an instance from Corollary~\ref{cor:target-F2}, with \(h\) injective.
	
	Choose the integral basis
	\[
	\mathcal C=(E_{11},E_{12},E_{21},E_{22}-E_{11})
	\]
	of \(V\). The first coordinate of \(X\in V\) with respect to
	\(\mathcal C\) is \(\operatorname{tr}(X)\). Indeed, if
	\[
	X=
	aE_{11}+bE_{12}+cE_{21}+d(E_{22}-E_{11}),
	\]
	then
	\(
	\operatorname{tr}(X)=a.
	\)
	
	For \(A,B\in\operatorname{SL}_2(\mathbb Z)\), let
	\(
	\Lambda_{\mathcal C}(A,B)\in\operatorname{SL}_4(\mathbb Z)
	\)
	denote the matrix, with respect to \(\mathcal C\), of the transformation
	\[
	X\longmapsto (A,B)\cdot X=AXB^{-1}.
	\]
	The coordinate vector of \(I_2\) with respect to \(\mathcal C\) is
	\(
	v=(2,0,0,1)^{\mathsf T}.
	\)

	Define
	\[
	\widehat\Lambda(A,B)
	=
	\begin{pmatrix}
		\Lambda_{\mathcal C}(A,B)&
		\Lambda_{\mathcal C}(A,B)v-v\\
		0&1
	\end{pmatrix}.
	\]
	If
	\(
	P=
	\begin{pmatrix}
		I_4&v\\
		0&1
	\end{pmatrix},
	\)
	then
		\(
	P^{-1}=
	\begin{pmatrix}
		I_4&-v\\
		0&1
	\end{pmatrix},
	\)
	and
	\[
	\widehat\Lambda(A,B)
	=
	P^{-1}
	\begin{pmatrix}
		\Lambda_{\mathcal C}(A,B)&0\\
		0&1
	\end{pmatrix}
	P.
	\]
	Hence, the mapping
	\(
	(A,B)\longmapsto\widehat\Lambda(A,B)
	\)
	is a homomorphism from 
	\(
	\operatorname{SL}_2(\mathbb Z)\times
	\operatorname{SL}_2(\mathbb Z)\) to
	$\operatorname{SL}_5(\mathbb Z).$
	
	Its restriction to
	\(\rho(F_2)\times\rho(F_2)\) is injective by
	Lemma~\ref{lem:matrix-diagonal-stabilizer}.
	
	The last column above the final entry of \(\widehat\Lambda(A,B)\) is the
	coordinate vector of
	\[
	(A,B)\cdot I_2-I_2
	=
	AB^{-1}-I_2
	\]
	with respect to \(\mathcal C\). Since the first coordinate in this basis is
	the trace, we obtain
	\begin{align*}
		\widehat\Lambda(A,B)_{1,5}
		&=
		\operatorname{tr}(AB^{-1}-I_2)\\
		&=
		\operatorname{tr}(AB^{-1})-2.
	\end{align*}
	
	For \(A,B\in\rho(F_2)\), the element \(AB^{-1}\) also belongs to
	\(\rho(F_2)\). By the choice of \(\rho\), every nonidentity element of
	\(\rho(F_2)\) has trace of absolute value greater than two. Hence
	\begin{equation}\label{eq:corner-equalizer}
		\widehat\Lambda(A,B)_{1,5}=0
		\quad\Longleftrightarrow\quad
		A=B.
	\end{equation}
	
	For every \(y\in Y\), put
	\(
	N_y=
	\widehat\Lambda\bigl(yg\rho,yh\rho\bigr)
	\)
	and let
	\(
	G=\langle N_y:y\in Y\rangle.
	\)
	Define
	\(
	\widehat\eta:F_Y\longrightarrow G
	\)
	by
	\[
	x\widehat\eta=
	\widehat\Lambda\bigl(xg\rho,xh\rho\bigr).
	\]
	As in the proof of Theorem~\ref{thm:matrix-stabilizer},
	\(\widehat\eta\) is injective. Therefore
	\(
	(N_y)_{y\in Y}
	\)
	i a free basis of \(G\).
	
	By~\eqref{eq:corner-equalizer}, for every \(x\in F_Y\),
	\[
	(x\widehat\eta)_{1,5}=0
	\quad\Longleftrightarrow\quad
	xg=xh.
	\]
	Since \(\widehat\eta\) is injective, \(G\) contains a nonidentity matrix
	with zero \((1,5)\)-entry if and only if \(\Eq(g,h)\) contains a
	nontrivial element. The result follows from
	Corollary~\ref{cor:target-F2}.
\end{proof}

By adjoining identity blocks,
Theorem~\ref{thm:matrix-stabilizer} remains valid in every dimension at
least four. By adjoining identity blocks and conjugating by a fixed
permutation matrix which fixes the first coordinate and sends the fifth
coordinate to the last,
Theorem~\ref{thm:matrix-corner} remains valid in every dimension at least
five. Thus the stabilizer theorem lowers the integral dimension bound from
nine to four, while the upper-right-corner theorem lowers the rational
dimension bound from nine to five and, at the same time, establishes
undecidability over the integers in the stronger case where the subgroup is free and the matrices have determinant $\pm 1$.

\section*{Acknowledgements}
The author thanks Jordi Delgado for finding several typos in a previous version of this paper. The author also thanks Vladimir Shpilrain and Enric Ventura for bibliographic references. The author was supported by national funds through the Funda\c{c}\~ao para a Ci\^encia e a Tecnologia, FCT, under the project UID/04674/2025.

	\bibliographystyle{plain}
\bibliography{Bibliografia}

\end{document}